\documentclass{amsart}
\usepackage{amsfonts}
\usepackage{amsmath,amscd}
\usepackage{amssymb}
\usepackage{amsthm}
\usepackage{newlfont}
\newcommand{\f}{\frac}

 \newtheorem{thm}{Theorem}[section]
 \newtheorem{cor}[thm]{Corollary}
 \newtheorem{lem}[thm]{Lemma}
 
 \theoremstyle{definition}
 
 \theoremstyle{remark}

 \numberwithin{equation}{section}

\begin{document}

\title[Structure of nilpotent Lie algebra by its multiplier]
 {Structure of nilpotent Lie algebra by its multiplier}

\author[P. Niroomand]{Peyman Niroomand}

\email{niroomand@dubs.ac.ir}
\address{School of Mathematics and Computer Science\\
Damghan University of Basic Sciences, Damghan, Iran}




\date{\today}


\begin{abstract} For a finite dimensional Lie algebra $L$, it is known that
$s(L)=\f{1}{2}(n-1)(n-2)+1-\mathrm{dim}~M(L)$ is
non negative. Moreover, the structure of all finite nilpotent Lie
algebras is characterized when $s(L)=0,1$ in \cite{ni,ni4}. In this
paper, we intend to characterize all nilpotent Lie algebra while
$s(L)=2.$
\end{abstract}

\maketitle

\section{Introduction}
The history of Schur multiplier of Lie algebras and the results on
it which are appeared in the recent three decades shows that why this
concept is the subject of works of several authors.

The Schur multiplier of Lie algebra , $M(L)$, have been appeared in
several papers since the beginning of 90's decade and some results on
it are obtained in \cite{ba, es, es2, es3, ha, mo,  sa, ya}. It is
of interest to obtain much more results on the Schur multiplier of
Lie algebras analogues to, $\mathcal{M}(G)$, the Schur multiplier of
group $($see \cite{kar}$)$.

Nilpotent Lie algebras take special attention of several authors, because of their similarity to
finite $p$-groups.

Let $L$ and $G$ be an $n$-dimensional nilpotent Lie algebra and a
finite $p$-group of order $p^n$, respectively. By a result of Green
\cite{gr} and Moneyhun \cite{mo},
\[t(G)=\f{1}{2}n(n-1)-\log_p|\mathcal{M}(G)|~\text{and}~
t(L)=\f{1}{2}n(n-1)-\mathrm{dim}~M(L)\] are  non negative integers.

 Analogues to the results of \cite{be, el, ni2, ni3, zh} which are obtained by  Berkovich,
  Ellis, the author and Zhou about characterizing $p$-group by the order of
its Schur multiplier when $t(G)=0,1,2,3,4,5$ the result of
\cite{es,es3, ha} are given a vast classification of finite
dimensional nilpotent Lie algebra for $t(L)\leq 8$.

 The author in \cite{ni1} gave a reduction of Green's bound for the Schur multiplier
 of non abelian $p$-group, and then he succeeded  to obtain  a similar result for
 a finite dimension  nilpotent Lie algebra in his joint paper \cite{ni}.
 More precisely, he proved the dimension of the Schur multiplier of an $n$-dimensional nilpotent Lie algebra is
 $\f{1}{2}(n-1)(n-2)+1-s(L)$, where $s(L)\geq 0$.
 The same question is remained "Can we characterize the structure of $L$ by $s(L)$?"
 The answer to this question  for $s(L)=0,1$ is positive.
 In fact, the author in \cite{ni} and \cite{ni4} showed that
 $s(L)=0$ if and only if $L\cong H(1)\oplus A(n-3)$ and $s(L)=1$ if and only if $L\cong L(4,5,2,4)$, where $H(m)$
 and $A(k)$ denote
 the Heisenberg and abelian  Lie algebra of dimension $2m+1$ and  $k$, respectively.
In the present paper, we answer to the question for $s(L)=2$.

 We seek two purposes by characterization of Lie algebras when $s(L)$ is in hand. The first one, we have a vast classification of
Lie algebras$($see \cite[Theorem 3.1]{ni} and Theorem \ref{mt}$)$. The second one, the characterization
of Lie algebras will be simplified by  $t(L)$ $($see \cite[Theorem 3.8]{ni4}$)$.

Throughout this paper all Lie algebras are of finite dimension and  $L$ is said to have the property $s(L)=2$ or briefly with $s(L)=2$, when $\mathrm{dim}~M(L)=\f{1}{2}(n-1)(n-2)-1$.
\section{Characterizing nilpotent Lie algebras with $s(L)=2$ }

In this section we intend to characterize all nilpotent Lie algebras with the property $s(L)=2$. First we introduce some useful results
which are needed in our proofs.
\begin{lem}$\mathrm{(See}$ \cite[Example 3]{es2} $\mathrm{and}$ \cite[Theorem 24]{mo}$\mathrm{).}$\label{h}
\begin{itemize}
\item[(i)]$\mathrm{dim}~(M(H(1)))=2$.
\item[(ii)]$\mathrm{dim}~(M(H(m)))=2m^2-m-1$ for all $m\geq 2$.
\end{itemize}
\end{lem}
\begin{cor}$\mathrm{(See}$ \cite[Corollary 2.3]{ni} $\mathrm{).}$\label{sr}
Let $L$ be a Lie algebra, $K$ a central ideal of
$L$ and $H=L/K$. Then
\[\mathrm{dim}~M(L)+\mathrm{dim}~(L^2\cap K)\leq \mathrm{dim}~M(H)+\mathrm{dim}~M(K)+\mathrm{dim}~(H/H^2\otimes K).\]
\end{cor}
\begin{thm}\label{ds}$\mathrm{(See}$ \cite[Theorem 1]{es}$\mathrm{).}$
Let $L_1$ and $L_2$ be two Lie algebras. Then
\[\mathrm{dim}~M(L_1\oplus L_2) = \mathrm{dim}~M(L_1) + \mathrm{dim}~M(L_2)
+ \mathrm{dim}~(L_1/L_1^2\otimes L_2/L_2^2).\]
\end{thm}

We are ready to obtain new results. Using notation of \cite{es3}, we have
\begin{thm}\label{t1}Let $L$ be an $n$-dimensional nilpotent Lie algebra with $s(L)=2$ and $\mathrm{dim}~L^2=2$. Then
\[\L\cong L(3,4,1,4)~\text{or}~L\cong L(4,5,2,4)\oplus A(1).\]
\end{thm}
\begin{proof}
First assume that $n=4$, \cite[Lemma 3.6]{ni4} implies that $L\cong L(3,4,1,4)$.

In the case $n\geq 5$, for the sake of clarity, we divide the proof into several parts relative to dimension of
$Z(L)$.

Let $\mathrm{dim}~Z(L)=4$ and $L^2\subseteq Z(L)$, there exists a central ideal $K$ of dimension
$2$ such that $L/K$ is not abelian.
By invoking \cite[Theorem 3.1]{ni}, we have $\mathrm{dim}~{M}(L)\leq \f{1}{2}(n-2)(n-5)$, and hence
Corollary \ref{sr} implies that $\mathrm{dim}~M(L)<\f{1}{2}(n-1)(n-2)-1$ which contradicts the assumption.
The case $L^2\nsubseteq Z(L)$ is obtained similarly.

Let $\mathrm{dim}~Z(L)=2$ and $L^2=Z(L)$.  Since $s(L)=2$, we may assume that  $n\geq 6$. The proof of \cite[Theorem 3.5]{ni4} shows there exists a central ideal
of dimension $1$ such that $L/K\cong H(m)\oplus A(n-2m-2)$ and $m\geq 2$. Owning to
 Lemma \ref{h}, Corollary \ref{sr} and Theorem \ref{ds}, we should have
 \[\mathrm{dim}~M(L)<\f{1}{2}(n-1)(n-2)-1,\] which is a contradiction. One can check that the dimension of Schur multiplier of $L$ is less than $\f{1}{2}(n-1)(n-2)-1$, when $L^2\neq Z(L)$ or $\mathrm{dim}~Z(L)=1$.

 In the final case, let $\mathrm{dim}~Z(L)=3$.  Since $s(L)=2$, it is readily shown that $L^2\subseteq Z(L)$.                                          Let $K$ be the complement $L^2$ in $Z(L)$, \cite[Theorem 3.1]{ni} asserts that $\mathrm{dim}~M(L/K)\leq \f{1}{2}(n-2)(n-3)$. On the other hand,
 Corollary \ref{sr} and the assumption follow that
\[\begin{array}{lcl}\f{1}{2}(n-1)(n-2)-1&=&\mathrm{dim}~M(L)\leq\mathrm{dim}~M(L/K)+\big(L/(K\oplus L^2)\big)\otimes K \vspace{.3cm}\\&=&
\mathrm{dim}~M(L/K)+n-3,\end{array}\]
which implies that $\mathrm{dim}~M(L)=\f{1}{2}(n-2)(n-3)$, hence $L/K\cong L(4,5,2,4)$ due to
\cite[Theorem 3.7]{ni4}, and so   $L\cong  L(4,5,2,4)\oplus K$.
\end{proof}
\begin{lem}\label{l1}There is no $n$-dimensional Lie algebra with the property $s(L)=2$ and $\mathrm{dim}~L^2\geq 3$.
\end{lem}
\begin{proof}Taking $k=3$ in \cite[Theorem 3.1]{ni}, we have \[\f{1}{2}(n-1)(n-2)-1=\mathrm{dim}~M(L)\leq \f{1}{2}(n+1)(n-4)+1,\] which is a contradiction.
\end{proof}
Using Theorem \ref{t1} and Lemma \ref{l1}, we may assume that $\mathrm{dim}~L^2=1$.
\begin{thm}Let $L$ be an $n$-dimensional Lie algebra and  $\mathrm{dim}~L^2=1$ with $s(L)=2$. Then $L\cong H(m)\oplus A(n-2m-1)$ and $m\geq 2$.
\end{thm}
\begin{proof}
Since $\mathrm{dim}~L^2=1$, \cite[Lemma 2.1]{ni4} implies that $L\cong H(m)\oplus A(n-2m-1)$.\
Using Lemma \ref{h} and Theorem \ref{ds} for all $m\geq 2$
\[\begin{array}{lcl}\mathrm{dim}~\big(M(H(m)\oplus A(n-2m-1))\big)&=&\f{1}{2}(2m^2-m-1)+(n-2m-1)(n-2m-2)\vspace{.3cm}\\&+&(n-2m-1)(2m)=\f{1}{2}n(n-3).\end{array}\]
A similar technic shows that  $\mathrm{dim}~\big(M(H(1)\oplus A(n-3))\big)=\f{1}{2}n(n-3)=2,$ as required.
\end{proof}
We summarize the results as follows.
\begin{thm} \label{mt}Let $L$ be an $n$-dimensional nilpotent Lie algebra with the property $s(L)=2$ if and only if $L$ is isomorphic to the  one of the  following Lie algebras.
\begin{itemize}
\item [(i)]$L(3,4,1,4)$;
\item[(ii)]$L(4,5,2,4)\oplus A(1)$;
\item[(iii)]$ H(m)\oplus A(n-2m-1)$ for all $m\geq 2$.
\end{itemize}
\end{thm}

\end{document}